\documentclass{article}

\usepackage{graphicx} 
\usepackage{amsmath,amsfonts,amssymb,amsthm,amscd}
\usepackage{datetime}
\usepackage{url}
\usepackage{hyperref}
\usepackage[usenames,dvipsnames]{xcolor}
\hypersetup{
    colorlinks,
    linkcolor={black!63!darkgray},
    citecolor={blue!70!white},
    urlcolor={blue!80!white}
}

\usepackage{geometry}
\newgeometry{top=0.75in, bottom=25mm, left=25mm, right=25mm, heightrounded, marginparwidth=0in, marginparsep=0.15in}

\usepackage{subcaption}
\numberwithin{figure}{section}
\numberwithin{table}{section}

\usepackage{framed} 
\usepackage{ulem}
\usepackage[T1]{fontenc}
\usepackage{pbsi}
\usepackage{mathabx}
\usepackage{dsfont}

\usepackage{enumitem}
\setlist[itemize]{leftmargin=0.65in}

\theoremstyle{plain} 
\newtheorem{theorem}{Theorem}

\numberwithin{theorem}{section}
\newtheorem*{theorem*}{Theorem}
\newtheorem*{conjecture*}{Conjecture}

\theoremstyle{definition} 

\newtheorem{remark}[theorem]{Remark}
\newtheorem{definition}[theorem]{Definition}

\newtheorem{ansatz}[theorem]{Assertion}

\theoremstyle{remark}
\newtheorem*{remark*}{Remark}

\renewcommand{\Re}{\operatorname{Re}}
\renewcommand{\Im}{\operatorname{Im}}

\newcommand{\mathtext}[1]{\text{\rm #1}}
\newcommand{\cf}{cf.~} 
\renewcommand{\emph}[1]{\textit{#1}}

\newcommand{\Iverson}[1]{\ensuremath{\left[#1\right]_{\delta}}}

\newcommand{\seqnum}[1]{\href{http://oeis.org/#1}{\color{ProcessBlue}{\underline{#1}}}}

\title{Picking up the partial sums of the M\"{o}bius function problem with probabilistic number theory}
\author{Dr. Maxie Dion Schmidt \\ \url{maxieds@gmail.com}}
\date{\small\it \today \ @\ \hhmmsstime{}}

\allowdisplaybreaks

\begin{document}

\maketitle

\begin{abstract} 
We revisit several hybrid multiplicative-to-additive type functions 
from a recent preprint article. 
These functions, $g(n)$ with Dirichlet generating function (DGF) 
$\zeta(s)^{-1} (1+P(s))^{-1}$ for $\Re(s) > 1$ where 
$P(s) = \sum_p p^{-s}$ is the prime zeta function, 
$|g(n)| = \lambda(n) g(n)$ with DGF 
$\zeta(2s)^{-1}(1-P(s))^{-1}$, 
and $C_{\Omega}(n)$ with DGF $(1-P(s))^{-1}$. 
Each of these function variants are defined in terms of the 
additive (respectively, strongly additive) 
functions $\omega(n)$ and $\Omega(n)$. These two auxiliary functions 
are used in the prior manuscript to relate partial sums of the 
classical M\"{o}bius function, $\mu(n)$, to signed partial sums involving 
the prime counting function, $\pi(x)$, and the Liouville lambda function, 
$\lambda(n) := (-1)^{\Omega(n)}$. 

In this article, we explore summing the identities from 
the first manuscript using several probabilistic assumptions about the 
independence of the values of $\Omega(n)$ and $\mu^2(n)$ for 
$n \leq x$ at large $x$. 
We recover proofs of the limiting asymptotic growth of $|M(x)| / \sqrt{x}$ 
whose hypotheses promise to be substantially more attainable to make 
rigorous than past results from other authors relying on the 
Riemann Hypothesis or assumption of the linear independence of the simple, 
non-trivial zeros of $\zeta(s)$. 

\textbf{Mathematics subject classifications (MSC 2010):} 
\emph{11A25; 11N37; 11N60; 11N64; and 11K65.} 
\end{abstract}

\section{Introduction}

\subsection{Preliminaries: The Mertens function}

\subsubsection{Definitions}

\begin{definition}[Canonical additive number theoretic functions]
\label{def_PrimeOmegaFuncs_v1}
For integers $n \geq 2$, we define the strongly and 
completely additive functions, respectively, 
that count the number of prime divisors of $n$ by 
\begin{align*}
\omega(n) & := \sum_{p|n} 1, \mathrm{\ and\ } 
\Omega(n) := \sum_{p^{\alpha} \mid\mid n} \alpha. 
\end{align*}
That is, if $n = p_1^{\alpha_1} \times \cdots p_r^{\alpha_r}$ is the 
factorization of $n$ into powers of distinct primes, then 
$\omega(n) = r$ and $\Omega(n) = \alpha_1 + \cdots + \alpha_r$. 
We use the convention that the functions $\omega(1) = \Omega(1) = 0$. 
\end{definition}

\begin{definition}
The M\"{o}bius function is the multiplicative function defined by 
\cite[\seqnum{A008683}]{OEIS}
\[
\mu(n) := \begin{cases} 
	1, & \text{ if $n = 1$}; \\ 
	(-1)^{\omega(n)}, & \text{ if $n \geq 2$ and $\omega(n) = \Omega(n)$ (i.e., if $n$ is squarefree);} \\ 
	0, & \text{ otherwise.}
        \end{cases}
\]
The Mertens function is defined to be the partial sum 
\cite[\seqnum{A002321}]{OEIS} 
\begin{align} 
M(x) & := \sum_{n \leq x} \mu(n), \mathrm{\ for\ } x \geq 1. 
\end{align} 
The Liouville lambda function is the completely multiplicative function 
defined for all $n \geq 1$ by $\lambda(n) := (-1)^{\Omega(n)}$ 
\cite[\seqnum{A008836}]{OEIS}. 
We have that $\lambda(n) = \mu(n)$ whenever $n \geq 1$ is squarefree. 
\end{definition}

\subsubsection{Analytic formulas, upper bounds and 
               conjectures on the limiting behavior of $M(x)$} 
\label{subSection_Intro_Mx_properties} 

A conventional approach to evaluating the limiting asymptotic 
behavior of $M(x)$ for large $x \rightarrow \infty$ considers an 
inverse Mellin transformation of the reciprocal of the Riemann zeta function. 
In particular, since 
\[
\frac{1}{\zeta(s)} = \prod_{p} \left(1 - \frac{1}{p^s}\right) = 
    \int_1^{\infty} \frac{s M(x)}{x^{s+1}} dx, \text{\ for\ } \Re(s) > 1, 
\]
we obtain by inversion that 
\[
M(x) = \lim_{T \rightarrow \infty}\ \frac{1}{2\pi\imath} \int_{T-\imath\infty}^{T+\imath\infty} 
     \frac{x^s}{s \zeta(s)} ds. 
\] 
The previous two representations lead to the next 
exact expression for $M(x)$ at any real $x > 0$. 
\nocite{TITCHMARSH} 

\begin{theorem}[An Analytic Formula for $M(x)$, Titchmarsh] 
\label{theorem_MxMellinTransformInvFormula} 
Assuming the Riemann Hypothesis (RH), there exists an infinite sequence 
$\{T_k\}_{k \geq 1}$ satisfying $k \leq T_k \leq k+1$ for each $k$ 
such that for any real $x > 0$ 
\begin{align*}
M(x) & = \lim_{k \rightarrow \infty} 
     \sum_{\substack{\rho: \zeta(\rho) = 0 \\ |\Im(\rho)| < T_k}} 
     \frac{x^{\rho}}{\rho \zeta^{\prime}(\rho)} - 2 + 
     \sum_{n \geq 1} \frac{(-1)^{n-1}}{n (2n)! \zeta(2n+1)} 
     \left(\frac{2\pi}{x}\right)^{2n} + 
     \frac{\mu(x)}{2} \Iverson{x \in \mathbb{Z}^{+}}. 
\end{align*} 
\end{theorem} 

A historical unconditional bound on the Mertens function due to 
Walfisz (circa 1963) 
states that there is an absolute constant $C_1 > 0$ such that 
\[
M(x) \ll x \cdot \exp\left(-C_1 \cdot \log^{\frac{3}{5}}(x) 
     (\log\log x)^{-\frac{3}{5}}\right). 
\]
Under the assumption of the RH, Soundararajan improved estimates 
bounding $M(x)$ from above for large $x$ and any $\epsilon > 0$ 
in the following form \cite{SOUND-MERTENS-ANNALS}: 
\begin{align*} 
M(x) & = O\left(\sqrt{x} \cdot \exp\left( 
     (\log x)^{\frac{1}{2}} (\log\log x)^{\frac{5}{2}+\epsilon}\right)\right). 
\end{align*} 

The RH is in fact equivalent to showing that 
$M(x) = O\left(x^{\frac{1}{2}+\epsilon}\right)$ for any 
$0 < \epsilon < \frac{1}{2}$. 
There is a rich history to the original statement of the 
\emph{Mertens conjecture} which 
asserts that for some absolute constant $C_2 > 0$, 
\[ 
|M(x)| < C_2 \sqrt{x}. 
\] 
The conjecture was first 
verified by Mertens himself to hold when $C_2 = 1$ for all $x < 10000$ 
without the benefit of modern computation. 
Since its beginnings in 1897, the Mertens conjecture was disproved 
by computational methods utilizing 
non-trivial simple zeta function zeros with comparatively small imaginary parts in the famous paper by 
Odlyzko and t\'{e} Riele \cite{ODLYZ-TRIELE}. 
More recent attempts 
at bounding $M(x)$ naturally consider determining the rates at which the 
oscillatory function $q(x) := M(x) / \sqrt{x}$ 
grows with (or without) bound along infinite 
subsequences, e.g., considering the asymptotics of $q(x)$ 
in the limit supremum and limit infimum senses. 

It is verified by computation 
that \cite[\cf \S 4.1]{PRIMEREC} 
\cite[\cf \seqnum{A051400}; \seqnum{A051401}]{OEIS} 
\[
M_{+} := 
     \limsup_{x\rightarrow\infty} \frac{M(x)}{\sqrt{x}} > 1.060\ \qquad \ \ \ \ 
     (\text{more recently: } M_{+} \geq 1.826054), 
\] 
and 
\[ 
M_{-} := \liminf_{x\rightarrow\infty} \frac{M(x)}{\sqrt{x}} \ \ \ \ 
     < -1.009\ \qquad\ \ (\text{more recently: } M_{-} \leq -1.837625). 
\] 
Based on work by Odlyzyko and t\'{e} Riele, it seems likely that 
each of these limits should evaluate to $\pm \infty$, respectively 
\cite{ODLYZ-TRIELE,MREVISITED,ORDER-MERTENSFN,HURST-2017}. 
A famous conjecture due to Gonek asserts that in fact the absolute value of 
$M(x)$ is constant scaled as in the following limit \cite{NG-MERTENS}:
\[
\limsup_{x \rightarrow \infty} \frac{|M(x)|}{ 
     \sqrt{x} (\log\log\log x)^{\frac{5}{4}}} < +\infty.
\]

\subsubsection{Remarks on probabilistic underpinnings: 
               Randomized models of the M\"obius function}

Some natural probabilistic models of the 
M\"obius function lead us to consider the behavior of $M(x)$ 
as a sum of independent and identically distributed (iid) random variables. 
Suppose that $\{M_n\}_{n \geq 1}$ is a sequence of iid 
$\{0,\pm 1\}$-valued random variables 
such that for all $n \geq 1$ 
$$\mathbb{P}[M_n = -1] = \mathbb{P}[M_n = +1] = \frac{3}{\pi^2}, 
  \mathtext{ and } \mathbb{P}[M_n = 0] = 1 - \frac{6}{\pi^2}.$$ 
That is, the sequence $\{M_n\}_{n \geq 1}$ provides a randomized model of the 
values of $\mu(n)$ on average. 
We may approximate limiting properties of the partial sums as 
$M(x) \cong \widebar{M}_x$ where 
$\widebar{M}_x := \sum_{n \leq x} M_n$ for $x \geq 1$. 
This viewpoint models predictions of certain limiting asymptotic behavior of the 
Mertens function. 
Under this probabilistic approximation to the Mertens function, 
we can compute the following results 
\cite[Theorem 9.4; \S 9]{BILLINGSLY-PROB-AND-MEASURE-BOOK} 
\cite[\cf \S 1]{ODLYZ-TRIELE}: 
\[
\mathbb{E}\left[\widebar{M}_x\right] \sim 0, 
     \operatorname{Var}\left[\widebar{M}_x\right] \sim \frac{6x}{\pi^2}, 
     \mathtext{ and } 
     \limsup_{x \rightarrow \infty} 
     \frac{\left\lvert \widebar{M}_x\right\rvert}{\sqrt{x \log\log x}} = 
     \frac{2\sqrt{3}}{\pi} \approx 1.1026578 
     \mathtext{ (almost surely).} 
\]
The limiting constant on the right-hand-side of the previous equation 
follows by an application of the law of the iterated logarithm in probability theory.

\subsection{Revisiting functions with Erd\'{o}s-Kac-type 
            central limit theorems}
\label{subSection_Restating_MDS-MERTENS_article_results}

\subsubsection{A somewhat lucky strategy}
The next definitions and identities for key auxiliary functions stated in 
this section are used to expand new exact formulas for the 
Mertens function in the reference. These preliminary results are 
summarized from the proofs in the 2022--2026 preprint article 
\cite{MDS-MERTENS}. 
Additive functions such as the prime omega function 
variants in Definition \ref{def_PrimeOmegaFuncs_v1} 
are special to work with because they often come 
bundled with central limit theorems characterizing the distribution of 
these functions for $n \leq x$ as $x \rightarrow \infty$. 
The curious and noteworthy convolution identity that 
$\chi_{\mathbb{P}} = \omega \ast \mu$ is simply too good to pass up as it allows 
us to relate the partial sums, $M(x)$, to signed additive function variants and 
the distribution of the prime numbers whose indicator function is 
$\chi_{\mathbb{P}}$ by applying inversion formulas for partial sums of 
convolutions \cite[\cf Theorem 1.3]{MDS-MERTENS}. 
We win, and seemingly win big, here in the next sections by defining the 
auxiliary component functions, 
$g(n)$ and $C_{\Omega}(n)$, and their 
partial sums in terms of the canonical additive number theoretic functions. 

\subsubsection{Recalling key auxiliary functions and results} 

\begin{definition}[Dirichlet convolution and Dirichlet inverse functions]
For any arithmetic functions $f$ and $h$, we define their 
Dirichlet convolution at $n$ by the divisor sum 
\[
(f \ast h)(n) := \sum_{d|n} f(d) h\left(\frac{n}{d}\right), 
     \mathrm{\ for\ } n \geq 1.
\]
The arithmetic function $f$ has a unique inverse with respect to Dirichlet convolution, 
denoted by $f^{-1}$, if and only if $f(1) \neq 0$. 
When it exists, the Dirichlet inverse of $f$ satisfies 
$\varepsilon(n) := (f \ast f^{-1})(n) = (f^{-1} \ast f)(n) = \delta_{n,1}$, the 
multiplicative identity function with respect to Dirichlet convolution, which is 
one if and only if $n = 1$ and is zero-valued otherwise. 
\end{definition}

\begin{definition}
The Dirichlet invertible function $g(n)$
is defined in terms of the shifted additive function 
$\omega(n)$ that counts the 
number of distinct prime factors of $n$ without multiplicity. 
In particular, motivated by the convolution identity that 
$\chi_{\mathbb{P}} = \omega \ast \mu$ is the 
characteristic function of the primes, 
we define the Dirichlet inverse function \cite[\seqnum{A341444}]{OEIS} 
\begin{equation}
\label{eqn_gInvn_def_v1}
g(n) := (\omega + \mathds{1})^{-1}(n), \mathrm{\ for\ } n \geq 1. 
\end{equation}
The inverse function in equation \eqref{eqn_gInvn_def_v1} 
is computed recursively by applying the formula \cite[\S 2.7]{APOSTOLANUMT}
\[
g(n) = \begin{cases}
	1, & \text{ if $n = 1$; } \\ 
	-\sum\limits_{\substack{d|n \\ d> 1}} \left(\omega(d) + 1\right) g\left(\frac{n}{d}\right), & 
	\text{ if $n \geq 2$. }
        \end{cases}
\]
\end{definition}
The motivation for the next definition of the arithmetic function, 
$C_{\Omega}(n)$, expanded below is that for all $n \geq 1$ 
\begin{equation*}
\lambda(n) C_{\Omega}(n) = (g \ast 1)(n) = \left( 
     \chi_{\mathbb{P}} + \varepsilon\right)^{-1}(n), 
\end{equation*}
where $\varepsilon(n) := \delta_{n,1}$ is the identity function with respect to 
Dirichlet convolution. 

\begin{definition}
Let the auxiliary component function, $C_{\Omega}(n)$, be defined for all 
$n \geq 1$ by the variant multinomial coefficient expansion 
\begin{equation}
\label{eqn_COmegan_product_exp_def}
C_{\Omega}(n) := (\Omega(n))! \times \prod_{p^{\alpha} || n} \frac{1}{\alpha!}. 
\end{equation}
\end{definition}

\begin{theorem}
\label{theorem_gn_key_idents}
The notation $|g(n)| = \lambda(n) g(n)$ denotes the absolute value of $g(n)$. 
We restate the following identities for the signed and unsigned magnitude 
of the function $g(n)$ for all $n \geq 1$:
\begin{subequations}
\begin{align}
\lambda(n) g(n) & = 
     \sum_{d|n} C_{\Omega}(d) \mu^2\left(\frac{n}{d}\right), \\ 
g(n) & = \sum_{d|n} \lambda(d) C_{\Omega}(d) \mu\left(\frac{n}{d}\right). 
\end{align}
Moreover, for any squarefree $n \geq 1$, we have that 
\begin{equation}
g(n) = (-1)^{\omega(n)} \times \sum_{0 \leq m \leq \omega(n)} 
     \binom{\omega(n)}{m} \times m!. 
\end{equation}
\end{subequations}
\end{theorem}

\begin{definition}
The summatory function of $g(n)$ is defined as follows 
\cite[\seqnum{A341472}]{OEIS}: 
\begin{equation}
\label{eqn_GInvx_PartialSumForms_v1} 
G(x) := \sum_{n \leq x} g(n), \mathtext{ for } x \geq 1. 
\end{equation} 
\end{definition}

\begin{theorem}
The Mertens function, $M(x)$, 
is related to the summatory function, $G(x)$, and 
the prime counting function, $\pi(x)$, through the following identities: 
\begin{subequations}
\begin{align}
\label{eqn_Mx_Gx_PartialSumGxOverp}
M(x) & = G(x) + \sum_{k \leq x} g(k) 
     \pi\left(\left\lfloor \frac{x}{k} \right\rfloor\right), \\ 
\label{eqn_Mx_Gx_PartialSumGxOverp_v2}
M(x) & = G(x) + \sum_{p \leq x} G\left(\left\lfloor \frac{x}{p} \right\rfloor\right). 
\end{align}
\end{subequations}
\end{theorem}

The Dirichlet generating function (DGF) of $g(n)$ is $\zeta(s)^{-1} (1+P(s))^{-1}$ 
for $\Re(s) > 1$ where $P(s) = \sum_p p^{-s}$ is the prime zeta function. 
The unsigned function $|g(n)|$ has 
DGF $\zeta(2s)^{-1}(1-P(s))^{-1}$ 
and the related component function $C_{\Omega}(n)$ has DGF 
$(1-P(s))^{-1}$ for $\Re(s) > 1$. 
Asymptotic formulae for each of these three respective functions 
are difficult to approach directly via traditional analytic methods 
because of the natural boundary of $P(s)$ at the line $\sigma = 0$. 

\subsection{The independence of 
            special multiplicative functions}
\label{subSection_RemarksOnIndependence}

\begin{ansatz}[Independence of certain arithmetic functions]
\label{ansatz_Independence_props_v1}
The following properties hold: 
\begin{itemize}
\item[(IH-A)] 
      For any $n, k \leq x$, as $x \rightarrow \infty$, we 
      have that the following holds: 
      \[
      \mathbb{P}\left(\mu^2(n) \neq 0 \ |\ \Omega(n) = k\right) = 
       \mathbb{P}\left(\mu^2(n) \neq 0\right).
      \]
\item[(IH-B)] 
           The $n \geq 1$ for which $\lambda(n) \mu^2(n) \neq 0$ are squarefree. 
           Hence, for such $n$, $\Omega(n) = \omega(n)$. 
           These squarefree $n \leq x$ have 
           a limiting asymptotic density, $d_0$ (defined as below), 
           as $x \rightarrow \infty$.
\item[(IH-C)] 
           If we select any $1 \leq n \leq x$ uniformly at random, then 
           $\mathbb{P}[\mu^2(n) \neq 0] = \frac{6}{\pi^2}$ as 
           $x \rightarrow \infty$. 
           That is to say that of the asymptotically $\frac{6x}{\pi^2}$ 
           squarefree integers $n \leq x$, it is equally likely 
           (with fixed weighted probability) that a randomly selected $n$ is 
           squarefree. This statement agrees with the average order of the 
           partial sums of $\mu^2(n)$ over $n \leq x$ being asymptotically 
           $\frac{6x}{\pi^2}$ for large $x$.

\end{itemize}
\end{ansatz}

\subsubsection{A few important remarks} 

\begin{remark}
Let $x \geq 1$ be large and suppose that $p \leq x$ is a fixed prime. 
Furthermore, suppose we label the integers $n \leq x$ as follows: 
$\mathcal{N}_x := \{n_1, n_2, \ldots, n_x\} = [1, x] \cap \mathbb{Z}$. 
Let the random variable $C_{p,x}$ be defined to be $\alpha$ if 
$p^{\alpha} || x$ for some $\alpha \geq 1$ and by zero if $p \nmid x$. 
If we select any fixed $N(x) \in \mathcal{N}_x$, 
i.e., we select a fixed integer in $[1, x]$ with 
uniform probability of $\frac{1}{x}$, then 
as $x \rightarrow \infty$, 
$C_{p,N(x)}$ converges in distribution to a geometric random variable
with fixed parameter $\frac{1}{p}$ 
\cite[\cf \S 1.2]{LOG-COMB-STRUCTS-BOOK}. 
That is, for any $k \geq 1$, we have that for large $x \rightarrow \infty$ 
\[
\mathbb{P}\left[C_{p,N(x)} = k\right] = \left(1 - p^{-1}\right) p^{-k}. 
\]
\end{remark}

\begin{remark}[Another key consequence of independence] 
\label{remark_IndepHypConsequence_for_COmegan_v1}
Since $\lambda(n)$ is completely multiplicative, i.e., we have that 
$\lambda(mn) = \lambda(m) \lambda(n)$ for any integers $m,n \geq 1$, 
we can use the identities for the function $g(n)$ restated in 
Theorem \ref{theorem_gn_key_idents} to show that 
\begin{equation}
\label{eqn_Gx_exp_starting_point_v1}
G(x) = \sum_{n \leq x} \lambda(n) C_{\Omega}(n) \times 
     \sum_{j \leq \left\lfloor \frac{x}{n} \right\rfloor} 
     \lambda(j) \mu^2(j), \text{\ for\ all\ } x \geq 1. 
\end{equation}
The function, $C_{\Omega}(n)$ whose distribution is central to 
bounding \eqref{eqn_Gx_exp_starting_point_v1}, 
depends only on the distribution of 
$\Omega(n) = k$ and the \emph{exponents} in the prime power 
factorization of $n$ into a product of powers of the $k$ 
distinct primes exactly dividing $n$ 
(\cf equation \eqref{eqn_COmegan_product_exp_def}). 
By our assertions on independence above, we see that the distribution of 
values of $C_{\Omega}(n)$ over $n \leq x$ are independent from whether 
any given $n \leq x$ is squarefree, 
i.e., that satisfies $\mu^2(n) \neq 0$. 
\end{remark} 

\subsubsection{Observations on (IH-A)}

The function $\mu^2$ is the characteristic function of the squarefree integers. 
For all $x \geq 1$, its summatory function satisfies \cite{HARDY-WRIGHT}
\begin{align*}
Q_{\operatorname{sq}}(x) & := \sum_{n \leq x} \mu^2(n) = 
     \frac{6x}{\pi^2} + O\left(\sqrt{x}\right), 
\end{align*}
though the error term may be improved by non-elementary methods. 
If $\mu^2(n) \neq 0$, i.e., if this $n \geq 1$ is squarefree, we find that 
$\lambda(n) = (-1)^{\omega(n)} = \mu(n)$. 
More to our point, we find that the following asymptotic densities of 
$\mu(n)$ taken over $n$ squarefree are asymptotically equally distributed 
between the two signed values $\mu(n) = \pm 1$: 
\begin{align*}
 & \lim_{x \rightarrow \infty}\ 
     \frac{\#\left\{n \leq x: \mu(n) = +1\right\}}{ 
     Q_{\operatorname{sq}}(x)} = 
     \lim_{x \rightarrow \infty}\ 
     \frac{\#\left\{n \leq x: \mu(n) = -1\right\}}{ 
     Q_{\operatorname{sq}}(x)} = \frac{1}{2}, \mathtext{ and } \\ 
 & \lim_{x \rightarrow \infty}\ \frac{1}{x} \times 
     \#\left\{n \leq x: \mu^2(n) \neq 0\right\} = \frac{6}{\pi^2}. 
\end{align*}
Thus, up to small, bounded error terms for finite $x$, it is equally likely 
that $\mu(n)$ assumes either $\pm 1$ conditioned on $n \leq x$ being squarefree. 


\subsubsection{Observations on (IH-B)}

The distribution of $\Omega(n) = k$ is regular and predictable for all 
$n \leq x$ at large $x$. If we define, the function $\mathcal{N}_m(x)$ for 
all $m \geq 0$ and $x \geq 2$ as \cite[\S 2.4]{MV} 
\[
\mathcal{N}_m(x) := \#\left\{n \leq x: \Omega(n) - \omega(n) = m\right\}, 
\] 
then we have that 
\[
\frac{\mathcal{N}_m(x)}{x} = d_m + O\left( 
     \left(\frac{3}{4}\right)^m \frac{(\log x)^{\frac{4}{3}}}{\sqrt{x}}\right). 
\]
The set $\{d_m\}_{m \geq 0}$ is a sequence of constants where each $d_m$ 
depends only on $m$. The result in the last equation is proved from the 
expansion of the next prime product generating function in the reference. 
\[
\mathcal{D}(z) = \sum_{m \geq 0} d_m z^m = \prod_p \left(1 - \frac{1}{p}\right) 
     \left(1 + \frac{1}{p - z}\right), \mathtext{\ \ for } \Re(z) \leq 1. 
\]

\section{Formulas for $M(x)$ given the independence hypotheses}
\label{Section_FormulasForMx_with_indep}

In this section we use the assumptions about independence hypothesized by  
Assertion \ref{ansatz_Independence_props_v1} 
to bound the oscillatory $G(x)$ by summing the formula expanded by 
equation \eqref{eqn_Gx_exp_starting_point_v1} below. 
We apply these new formulas bounding $G(x)$ from below at 
large $x$ to the identity in equation 
\eqref{eqn_Mx_Gx_PartialSumGxOverp_v2}. 
This approach yields the new bounds for $M(x)$ along large 
infinite monotone increasing subsequences of integers stated in 
Section \ref{subSection_BoundsOnMxUsingGx}. 

\subsection{Expanding the key formula for $G(x)$ from 
            \eqref{eqn_Gx_exp_starting_point_v1}}

\begin{remark}[Distribution of values of $\Omega(n)$] 
\label{remark_Omegan_eqk_density_formula_v1} 
We have uniformly for $1 \leq k \leq \log\log x$ that \cite[\S 7.4]{MV} 
\[
\frac{1}{x} \times \#\{3 \leq n \leq x: \Omega(n) = k\} = 
     \frac{(\log\log x)^{k-1}}{(\log x) (k-1)!} 
     + O\left(\frac{k}{(\log\log x)^2}\right), \mathtext{ as } x \rightarrow \infty. 
\] 
For $k := r \log\log x$ with $r > 1$, the density of the $\Omega(n) = k$ 
in the form of the previous equation 
are much smaller so that we can consider the sums over the 
intervals $k > \log\log x$ to contribute asymptotically 
negligible error terms. More precisely, for any $1 < r < 2$ we can prove that 
as $x \rightarrow \infty$ \cite[\cf Theorem 7.20]{MV}
\[
\frac{1}{x} \times \#\left\{3 \leq n \leq x: \Omega(n) > r \log\log x\right\} \ll 
     (\log x)^{r - 1 - r \log r}.
\]
Larger yet values of $\Omega(n)$ are found along the sequence of 
primorial integers \cite[\seqnum{A002110}]{OEIS}. 
In general, $\Omega(n) \leq \log_2 x$ for all $2 \leq n \leq x$. 
The distribution of $\Omega(n)$ at these comparatively rare large values of the 
function may be discarded in the analysis given below 
as they similarly yield negligible 
error terms compared to main terms that depend on $\Omega(n) = k$ for 
$1 \leq k \leq \log\log x$.
\end{remark}

\begin{subequations}
\begin{theorem}
For $1 \leq n \leq x$, we define the following sums: 
\begin{equation}
\widehat{Q}_{1,n}(x) := \sum_{j \leq x} \lambda(nj) \mu^2(j). 
\end{equation}
Suppose that the independence properties in the last sections hold. 
We have that 
\begin{align}
\widehat{Q}_{1,n}(x) & \sim 
     \frac{6x}{\pi^2} \cdot \frac{(-1)^{\lfloor \log\log x\rfloor}}{ 
     2 \sqrt{2\pi \log\log x}}. 
\end{align}
\end{theorem}
\begin{proof}
The function at hand is used to define the inner sum component of 
equation \eqref{eqn_Gx_exp_starting_point_v1}. 
We compute the sum as follows given the assertions on independence 
stated in Assertion \ref{ansatz_Independence_props_v1} 
\cite[Lemma B.3]{MDS-MERTENS}: 
\begin{align}
\widehat{Q}_{1,n}(x) & \sim \sum_{1 \leq k \leq \log\log x} 
     \frac{(-1)^k (\log\log x)^{k-1}}{ 
     (\log x) (k-1)!} \times Q_{\operatorname{sq}}\left(\frac{x}{n}\right). 
     \qedhere
\end{align}
\end{proof}
\end{subequations}

\begin{theorem} 
\label{theorem_Gx_lowerbounds_x2_v1}
We have that 
\begin{subequations}
\begin{align}
\label{eqn_Gx_prob_interpret_v4}
|G(x)| & \gg \frac{\sqrt{x}}{(\log\log\log x)^{\frac{3}{2}}} \\ 
\label{eqn_Gx_gg_v2_w_cx_1}
|G(x)| & \gg \frac{\sqrt{x\log\log x}}{(\log\log\log x)}.
\end{align}
\end{subequations}
\end{theorem}
\begin{proof}[Remarks]
We observe that the exponent in equation 
\eqref{eqn_Gx_prob_interpret_v4} 
also appears in the analysis of other 
upper bounds on $M(x)$ \cite[\cf Theorem 1]{NG-MERTENS}. That is, 
under certain conjectured hypotheses, 
Ng proved that $M(x) \ll \sqrt{x} (\log\log x)^{\frac{3}{2}}$. 
This coincidence is encouraging to find naturally appearing in the next 
proofs of the lower bound asymptotics of $G(x)$ which characterize  
key exact formulas for $M(x)$ from \cite{MDS-MERTENS}. 
\end{proof}
\begin{proof}[Proof of \eqref{eqn_Gx_prob_interpret_v4}]
We assume the value $\Omega(n) = k$ is equally likely to occur at any given 
$n \leq x$, or at any particular $1 \leq \frac{x}{n} \leq x$, 
according to the density formulas for $\Omega(n)$ summarized in 
Remark \ref{remark_Omegan_eqk_density_formula_v1}. 
This assumption allows us to take the average order of 
$\widehat{Q}_{1,n}\left(\frac{x}{n}\right)$, done by effectively 
inserting a scaling factor of $\frac{1}{x}$ which 
represents the average of the function summed over $n \leq x$, 
to sum \eqref{eqn_Gx_exp_starting_point_v1}. 
Furthermore, because $C_{\Omega}(n)$ is a multinomial theorem 
coefficient variant depending 
only on $\Omega(n) := k$ and the distinct prime powers dividing $n$, 
we find that 
(\cf Remark \ref{remark_IndepHypConsequence_for_COmegan_v1}) 
\begin{subequations}
\begin{align}
G(x) & = \sum_{n \leq x} C_{\Omega}(n) 
     \widehat{Q}_{1,n}\left(\frac{x}{n}\right) && \\ 
     & = \frac{1}{x} \times \sum_{n,k \leq x} \pi\left(x^{1/k}\right)^k 
     \widehat{Q}_{1,n}\left(\frac{x}{n}\right) && 
     \mathtext{\ \ (on average, multinomial theorem)} \\ 
     & \sim \frac{3x (-1)^{\lfloor \log\log x\rfloor}}{\sqrt{2\pi^5 \log\log x}} \times 
     \sum_{k \leq x} \frac{k^k (\log\log x)^{k-1}}{(\log x)^{k} (k-1)!} && \\ 
     & \gg \frac{x}{\sqrt{\log\log x}} \times 
     \sum_{k \leq x} \frac{(\log\log x)^{k-1} e^k}{(\log x)^{k} \sqrt{k}}. && 
     \mathtext{\ \ (Stirling's approximation)} 
\end{align}
Let $b_x := \frac{(\log\log x)e}{\log x}$. As $x \rightarrow \infty$, 
we have that 
\begin{align}
\sum_{k \leq x} \frac{(\log\log x)^{k-1} e^k}{(\log x)^{k} \sqrt{k}} & \sim 
     \frac{\sqrt{\pi}}{(\log\log x) \log b_x} \operatorname{erf}\left( 
     \sqrt{x \log b_x}\right) \\ 
     & \sim \frac{1}{(\log\log\log x)}\left(\sqrt{\pi} - 
     \frac{e^{-x \log b_x}}{\sqrt{x \log b_x}}\right) \\ 
     & \gg \frac{\sqrt{\log\log x}}{\sqrt{x} (\log\log\log x)^{\frac{3}{2}}}. 
     \qedhere
\end{align}
\end{subequations}
\end{proof}
\begin{proof}[Proof of \eqref{eqn_Gx_gg_v2_w_cx_1}]
\label{remark_classical_tighter_bounds_v1}
By the mean value theorem \cite[\S 1.4]{NISTHB} and the 
Euler-Maclaurin summation formula, 
for some $\widehat{B}(x) \leq x$, we can write 
\[
\sum_{k \leq x} \frac{\sqrt{k} (\log\log x)^{k-1} e^k}{(\log x)^{k}} \sim 
     \widehat{B}(x) \times \sum_{k \leq x} \frac{(\log\log x)^{k-1} e^k}{ 
     (\log x)^{k} \sqrt{k}}, \mathtext{\ as\ } x \rightarrow\infty. 
\]
It is not difficult to prove that $\widehat{B}(x) \sim \log\log x$ as $x \rightarrow \infty$. 
In fact, repeating the steps above, we find that
\begin{equation}
\label{eqn_Gx_sim_and_second_lower_bound_v1}
|G(x)| \sim \frac{3x \sqrt{\log\log x}}{\sqrt{2\pi^5} (\log\log\log x)} \left( 
     \sqrt{\pi} - 
     \frac{e^{-x \log b_x}}{\sqrt{x \log b_x}}\right) 
     \gg \frac{\sqrt{x \log\log x}}{(\log\log\log x)}. 
\end{equation}

We cannot tighten the lower bound 
more than in the equation above 
(\cf Remark \ref{remark_classical_tighter_bounds_v2}). 
\end{proof}

\subsection{Asymptotic bounds on $M(x)$ with applications to classical conjectures}
\label{subSection_BoundsOnMxUsingGx}

We will use the asymptotic formulas for $G(x)$ 
established in the last section with the prime sum identity that 
exactly defines $M(x)$ in \eqref{eqn_Mx_Gx_PartialSumGxOverp_v2}. 
The Abel summation formula and Euler-Maclaurin summation bounds and 
integration by parts shows that 
\begin{subequations}
\begin{align}
\label{eqn_Mx_prob_Gx_interpret_v5}
M(x) & \sim G(x) + x \times \int G^{\prime}(x) \frac{dx}{x} \sim 2G(x).
\end{align}
\end{subequations}

\begin{subequations}

\begin{theorem}[Approaching classical conjectures]
\label{theorem_Mx_classical-type_bounds_v1}
There are absolute constants $A_1, A_2 > 0$ such that
\begin{equation}
\limsup_{x \rightarrow \infty} \frac{|M(x)| 
     (\log\log\log x)^{\frac{3}{2}}}{\sqrt{x}} \in \left\{A_1, +\infty\right\}, 
\end{equation}
and 
\begin{equation}
\limsup_{x \rightarrow \infty} \frac{|M(x)| 
     (\log\log\log x)}{\sqrt{x \log\log x}} \ \ \in \left\{A_2, +\infty\right\}, 
\end{equation}
\end{theorem}

\begin{remark}[More precise corollaries]
\label{remark_classical_tighter_bounds_v2}
We expect that the first limit in 
Theorem \ref{theorem_Mx_classical-type_bounds_v1} is in fact unbounded 
growing to $+\infty$. We also expect that the 
second limit evaluates to the bounded absolute constant, $A_2 > 0$, 
rather than growing unboundedly. 
Our argument follows from the 
proof of \eqref{eqn_Gx_gg_v2_w_cx_1} and the asymptotic expansion of 
the error function as $x \rightarrow \infty$. 
In particular, we can show that 
\[
G(x) \gg \frac{\sqrt{x \log\log x}}{\log\log\log x}. 
\]
For any strictly increasing $f(x) \rightarrow +\infty$ as 
$x \rightarrow \infty$, we have
\[
\frac{\sqrt{x \log\log x} f(x)}{\log\log\log x} \gg 
     \frac{\sqrt{x \log\log x}}{\log\log\log x}. 
\]
The reverse strict lower bound similarly holds if we replace $f$ by a 
(non-constant) monotone decreasing $g(x) = o(1)$ as $x \rightarrow \infty$.
Hence, in fact we have boundedness of the second limit in the following form 
(\cf equation \eqref{eqn_Gx_sim_and_second_lower_bound_v1} to obtain the 
precise constant value): 
\begin{equation}
\label{eqn_LimsupMx_MDS_eq_constant_v2}
\limsup_{x \rightarrow \infty} \frac{|M(x)| (\log\log\log x)}{ 
     \sqrt{x \log\log x}} = 
     \frac{6}{\sqrt{2 \pi^5}} \approx 0.242528 < +\infty. 
\end{equation}
We compare the result in the previous equation to the limiting scaling function 
given by Gonek's famous conjecture stated in the introduction which is 
hypothesized to evaluate precisely to a constant limit. 
That is, the scaling function of $|M(x)|$ in equation 
\eqref{eqn_LimsupMx_MDS_eq_constant_v2} 
captures the limiting growth of the 
unboundedness of the Mertens function precisely up to a constant. 
\end{remark}

\begin{remark}[Most likely, in all probability, the result that 
               takes the classically conjectured cake]
Moreover, and finally, 
because $\log\log x \gg \log\log\log x$ as $x \rightarrow \infty$, the 
truth of equation \eqref{eqn_LimsupMx_MDS_eq_constant_v2} 
implies that the classical 
unboundedness conjecture holds in the following form:
\begin{equation}
\limsup_{x \rightarrow \infty} \frac{|M(x)|}{\sqrt{x}} = +\infty. 
\end{equation}
If we keep the signed terms in the proofs of 
Theorem \ref{theorem_Gx_lowerbounds_x2_v1}, 
we should also be able to recover both of the signed versions of the 
unboundedness of $M(x)/\sqrt{x}$. 
We remark that Ingham's 
proof of these two results published in 1942 on the limiting growth of 
$M(x) / \sqrt{x}$ towards unboundedness, respectively as $\pm \infty$ in the 
limit supremum and limit infimum senses, 
is conditioned on conjectured hypotheses about the linear independence 
of the non-trivial simple zeros of $\zeta(s)$ \cite{INGHAM}. 
\end{remark}

\end{subequations}

\section{Conclusions}

We extended results in \cite{MDS-MERTENS} that 
relate the partial sums of the M\"{o}bius function, $M(x)$, 
to arithmetic functions that 
have deep connections between the distributions of the 
canonically additive functions, $\omega(n)$ and $\Omega(n)$, as well as the 
prime counting function, $\pi(x)$. 
Given reasonable probabilistic assumptions about the 
independence of the events that 
$\Omega(n) = k$ and $\mu^2(n) \neq 0$ for $n \leq x$ and $k \leq \log\log x$ as 
$x \rightarrow \infty$, we proved variants of the classical conjectures on the 
Mertens function from the introduction as 
Theorem \ref{theorem_Mx_classical-type_bounds_v1} 
(\cf Remark \ref{remark_classical_tighter_bounds_v2}). 
The approach to probabilistic assumptions in this manuscript may hopefully 
be made fully rigorous by other methods, 
say accurate to within some small error terms that vanish or are 
negligible as $x \rightarrow \infty$, 
so that we may treat Theorem \ref{theorem_Mx_classical-type_bounds_v1} 
to hold absolutely without the independence hypotheses.

The logarithm of $C_{\Omega}(n)$ satisfies an 
Erd\'{o}s-Kac-type central limit theorem that converges in distribution to the 
smooth standard normal CDF, 
$\Phi(z)$, for $n \leq x$ as $x \rightarrow \infty$ where 
\cite[\cf Theorem 1.8]{MDS-MERTENS} 
\[
\Phi(z) := \frac{1}{\sqrt{2\pi}} \times \int_{-\infty}^z e^{-\frac{t^2}{2}} dt, 
     \mathrm{\ for\ any\ } z \in (-\infty, \infty). 
\]
Central limit type theorems in analog to the famous Erd\'{o}s-Kac theorem 
for $\omega(n)$, or often times for $\Omega(n)$, 
reflect a probabilistic approach to studying the 
distributions of many (typically additive) number theoretic functions that 
satisfy usually reasonable properties over the primes 
\cite{ELLIOTT-V1,ELLIOTT-V2,TENENBAUM-PROBNUMT-METHODS}. 
This approach reflects Kac's original intuition 
about the spread of $\omega(n)$  
which was subsequently made formally rigorous 
by Erd\'{o}s who used sieves to formalize his proof of the theorem 
\cite{BILLINGSLY-CLT-PRIMEDIVFUNC,ERDOS-KAC-REF}. 
Certain Erd\'{o}s-Kac-type central limit theorems for arithmetic 
functions converge in distribution as $x \rightarrow \infty$ to a 
CDF that is not standard normal 
\cite{ELLIOTT-V1,ELLIOTT-V2,TENENBAUM-PROBNUMT-METHODS}. 
We surmise that the properties of an alternate probability distribution 
may be used to alter or vary asymptotic bounds we obtained in this article. 

%


\renewcommand{\em}{\it}


\end{document}